\documentclass[11pt]{amsart}
\usepackage{amssymb,amsmath,epsfig,mathrsfs, enumerate, xparse, mathtools}
\usepackage[pagewise]{lineno}
\usepackage[nodisplayskipstretch]{setspace}
\usepackage{graphicx}\usepackage[normalem]{ulem}
\usepackage{fancyhdr}
\usepackage[margin=2.5cm]{geometry}
\usepackage{hyperref}
\hypersetup{
colorlinks   = true,
urlcolor     = blue,
linkcolor    = blue,
citecolor   = red ,
bookmarksopen=true
}

\theoremstyle{plain}
\newtheorem{thrm}{Theorem}[section]
\newtheorem{lemma}[thrm]{Lemma}
\newtheorem{prop}[thrm]{Proposition}

\allowdisplaybreaks
\begin{document}
\newcommand{\sn}{\mathbb{S}^{n-1}}
\newcommand{\SL}{\mathcal L^{1,p}( D)}
\newcommand{\Lp}{L^p( Dega)}
\newcommand{\CO}{C^\infty_0( \Omega)}
\newcommand{\Rn}{\mathbb R^n}
\newcommand{\Rm}{\mathbb R^m}
\newcommand{\R}{\mathbb R}
\newcommand{\Om}{\Omega}
\newcommand{\Hn}{\mathbb H^n}
\newcommand{\A}{\alpha }
\newcommand{\B}{\beta}
\newcommand{\eps}{\ve}
\newcommand{\BVX}{BV_X(\Omega)}
\newcommand{\p}{\partial}
\newcommand{\IO}{\int_\Omega}
\newcommand{\bG}{\boldsymbol{G}}
\newcommand{\bg}{\mathfrak g}
\newcommand{\bz}{\mathfrak z}
\newcommand{\bv}{\mathfrak v}
\newcommand{\Bux}{\mbox{Box}}
\newcommand{\e}{\ve}
\newcommand{\X}{\mathcal X}
\newcommand{\Y}{\mathcal Y}
\newcommand{\la}{\lambda}
\newcommand{\vf}{\varphi}
\newcommand{\rhh}{|\nabla_H \rho|}
\newcommand{\Ba}{\mathcal{B}_\beta}
\newcommand{\Za}{Z_\beta}
\newcommand{\ra}{\rho_\beta}
\newcommand{\n}{\nabla}
\newcommand{\vt}{\vartheta}
\newcommand{\its}{\int_{\{y=0\}}}
\newcommand{\py}{\partial_y^a}

\numberwithin{equation}{section}

\newcommand{\RN} {\mathbb{R}^N}
\newcommand{\Sob}{S^{1,p}(\Omega)}
\newcommand{\Dxk}{\frac{\partial}{\partial x_k}}
\newcommand{\Co}{C^\infty_0(\Omega)}
\newcommand{\Je}{J_\ve}
\newcommand{\beq}{\begin{equation}}
\newcommand{\bea}[1]{\begin{array}{#1} }
	\newcommand{\eeq}{ \end{equation}}
\newcommand{\ea}{ \end{array}}
\newcommand{\eh}{\ve h}
\newcommand{\Dxi}{\frac{\partial}{\partial x_{i}}}
\newcommand{\Dyi}{\frac{\partial}{\partial y_{i}}}
\newcommand{\Dt}{\frac{\partial}{\partial t}}
\newcommand{\aBa}{(\alpha+1)B}
\newcommand{\GF}{\psi^{1+\frac{1}{2\alpha}}}
\newcommand{\GS}{\psi^{\frac12}}
\newcommand{\HFF}{\frac{\psi}{\rho}}
\newcommand{\HSS}{\frac{\psi}{\rho}}
\newcommand{\HFS}{\rho\psi^{\frac12-\frac{1}{2\alpha}}}
\newcommand{\HSF}{\frac{\psi^{\frac32+\frac{1}{2\alpha}}}{\rho}}
\newcommand{\AF}{\rho}
\newcommand{\AR}{\rho{\psi}^{\frac{1}{2}+\frac{1}{2\alpha}}}
\newcommand{\PF}{\alpha\frac{\psi}{|x|}}
\newcommand{\PS}{\alpha\frac{\psi}{\rho}}
\newcommand{\ds}{\displaystyle}
\newcommand{\Zt}{{\mathcal Z}^{t}}
\newcommand{\XPSI}{2\alpha\psi \begin{pmatrix} \frac{x}{\left< x \right>^2}\\ 0 \end{pmatrix} - 2\alpha\frac{{\psi}^2}{\rho^2}\begin{pmatrix} x \\ (\alpha +1)|x|^{-\alpha}y \end{pmatrix}}
\newcommand{\Z}{ \begin{pmatrix} x \\ (\alpha + 1)|x|^{-\alpha}y \end{pmatrix} }
\newcommand{\ZZ}{ \begin{pmatrix} xx^{t} & (\alpha + 1)|x|^{-\alpha}x y^{t}\\
	(\alpha + 1)|x|^{-\alpha}x^{t} y &   (\alpha + 1)^2  |x|^{-2\alpha}yy^{t}\end{pmatrix}}
\newcommand{\norm}[1]{\lVert#1 \rVert}
\newcommand{\ve}{\varepsilon}
\newcommand{\D}{\operatorname{div}}
\newcommand{\G}{\mathscr{G}}
\newcommand{\W}{\tilde{W}}

\title[carleman estimates etc ]{Carleman estimates for sub-Laplacians on Carnot groups}

\author{Vedansh Arya}
\address{Tata Institute of Fundamental Research\\
Centre For Applicable Mathematics \\ Bangalore-560065, India}\email[Vedansh Arya]{vedansh@tifrbng.res.in}

\author{Dharmendra Kumar}
\address{Tata Institute of Fundamental Research\\
Centre For Applicable Mathematics \\ Bangalore-560065, India}\email[Dharmendra Kumar]{dharmendra2020@tifrbng.res.in}



%
%
%

\maketitle
\begin{abstract}
	In this note, we  establish a new Carleman estimate  with singular weights  for the sub-Laplacian  on a  Carnot group $\mathbb G$  for  functions satisfying  the discrepancy assumption in \eqref{disc} below. We use such an estimate to derive a sharp vanishing order estimate   for solutions to stationary Schr\"odinger equations. 
	
\end{abstract}

\maketitle

\section{Introduction and Statement of the main result}
In this note, we give an elementary proof of  a $L^{2}-L^{2}$ type Carleman estimate with singular weights for the sub-Laplacian on Carnot groups. Using such an estimate, we  present a new  application to an upper bound on the maximal order of vanishing for solutions to stationary Schr\"odinger equations \eqref{e0}.  Such a result   as in Theorem \ref{main} below  constitutes a quantitative version of the  strong unique continuation  property and   can be thought of as a subelliptic  generalization of a similar  quantitative uniqueness  result due to  Bourgain and Kenig in \cite{BK}.   .

Concerning the question of interest in this note, the unique continuation property, we mention that for general uniformly elliptic equations there are essentially two known methods  for proving it. The former  is  based on Carleman inequalities, which are appropriate weighted versions of  Sobolev-Poincar\'e inequalities. This method  was first introduced by T. Carleman in his fundamental work \cite{C} in which  he  showed that  strong unique continuation holds  for equations of the type  $-\Delta u +V u = 0$, with $V \in L^{\infty}_{loc}(\R^2)$.
Subsequently, his estimates were generalised  in \cite{A} and \cite{AKS} to uniformly elliptic operators with $C^{2, \alpha}_{loc}$ and $C^{0,1}_{loc}$ principal part respectively in all dimensions. 
We recall that unique continuation fails in general when the coefficients of the principal part are only H\"older continuous, see  \cite{Pl}.
The second approach came up in the works of Lin and Garofalo, see \cite{GL1}, \cite{GL2}. Their method is based on the almost monotonicity of a generalisation of the frequency function, first introduced by Almgren in \cite{Al} for harmonic functions. Using this approach, they were able to obtain new quantitative  information  for the solutions to  divergence form elliptic equations with Lipschitz coefficients which in particular encompass and improve on those in \cite{AKS}. 

The unique continuation in subelliptic setting of a Carnot group is however much subtler  in the sense that  strong unique continuation property is in general not true for solutions to \eqref{e0}. This follows from some interesting work of Bahouri (\cite{Bah}) where the author showed that unique continuation is not true for  even smooth and  compactly supported perturbations of the sub-Laplacian. Subsequently in the setting of the Heseinberg group $\mathbb H^n$,  it is shown by Garofalo and Lanconelli in \cite{GLa} that if the solutions to \eqref{e0} additionally satisfy  the discrepancy assumption of the type \eqref{disc}, then the strong unique continuation holds. Such a result has been generalized to Carnot groups of arbitrary step in \cite{GR}. We also refer to the
 recent work \cite{G} where it is shown that in general, the Almgren type monotonicity fails even when $\mathbb G =\mathbb H^n$.  It is to be noted that the discrepancy condition \eqref{disc} trivially holds in the Euclidean case. See Section \ref{pr} below.

The purpose  of this note is to establish a new Carleman estimate  in  the framework of \cite{GR} where the strong unique continuation  is known so far using which  we  prove the vanishing order estimate in Theorem \ref{main} below.  We now state our main results.
\subsection{Statement of the main results}
Our first result is the subelliptic analogue of the well known Carleman estimate  in \cite{BK}.  See also \cite{EV}, \cite{Bk}. We refer to Section \ref{pr} for relevant notations and notions. 
\begin{thrm} \label{cest}
	Let   $w \in C^{2}_{0}(B_R \setminus \{e\})$ satisfy the discrepancy assumption in \eqref{disc} below for some $\delta>0$. Also assume that $V$ satisfies \eqref{vass}. Then there exists a universal constant $R_0>0$  depending on $\delta,$ $C_E$ and $Q$ such that  for all $R \leq R_0$  and $\A > CK^{2/3} +Q$, the following estimate holds 
	\begin{align}\label{est}
		&\alpha^3\int \rho^{-2\alpha- 4+\ve} w^2 e^{2\alpha \rho^{\ve}} \psi dg \leq  C \int \rho^{-2\alpha} e^{2\alpha \rho^{\ve}} (\Delta_{H} w+Vw)^2 \psi^{-1} dg,
	\end{align}
	for some universal $C$ and  $\ve= \delta/2.$
\end{thrm}
Using the Carleman estimate in Theorem \ref{cest} above, we derive the following quantitative uniqueness result for solutions to 
\begin{equation}\label{e0}
		- \Delta_{H} u =V u \;\;\text{in $B_R$},
	\end{equation}
where $V$ satisfies the following growth condition
	\begin{equation}\label{vass}
		|V| \leq K \psi.
	\end{equation}
	Since the regularity issues are not our main concern, we will assume  apriori that  $u, X_{i}u, X_{i}X_{j} u, Zu$   are in $L^{2}(B_R)$ with respect to the Haar measure $dg$. 
	
\begin{thrm}\label{main}
	Let $u$ be a solution to \eqref{e0} where $V$ satisfies \eqref{vass}. Futhermore assume that $u$ satisfies the discrepancy assumption in \eqref{disc} below.	
Then there exists a constant $C = C(Q,C_E, \delta)>0$ such that for all $r<R_0/8$, we have
	\begin{align}
		|| u \psi^{1/2}||_{L^2(B_r)} > Cr^{A},
	\end{align}
	where $A=CK^{2/3}+C+C\left(\left(1 + || u \psi^{1/2}||_{L^2(B_{R_0})}\right)\Big/||  u \psi^{1/2}||_{L^2(B_{R_0/4})}\right)^{4/3}$ and $R_0$ is as in the Theorem \ref{cest}.
\end{thrm}
It is worth emphasizing that, when $h=1$, from \eqref{defsi} we have $\psi \equiv 1$. In this case the constant $K$ in \eqref{vass} can be taken to be  $||V||_{L^{\infty}}$, and therefore Theorem \ref{main} reduces to the cited Euclidean result in \cite{BK} which is sharp in view of Meshov's counterexample in \cite{Me}. We also note that when $V$ satisfies the additional hypothesis
\[
|ZV| \leq K\psi,
\]
then, using a variant of the frequency function approach, the following sharper estimate was established in \cite{B} for solutions to \eqref{e0},
\begin{equation}\label{m}
||u||_{L^{\infty}(B_r)} \geq C_1 \left(\frac{r}{R_0}\right)^{C_2 (\sqrt{K}+1)}.   
\end{equation}

The reader should note that  for Laplacian on a compact manifold the counterpart of \eqref{m} was first obtained using Carleman estimates by Bakri in \cite{Bk}. This generalises the sharp vanishing order estimate of Donnelly and Fefferman in \cite{DF1,DF2} for eigenfunctions of the Laplacian. We also mention that, for the standard Laplacian, the result of Bakri was subsequently obtained by Zhu \cite{Zhu1}, using a variant of the frequency function approach in \cite{GL1, GL2}. This was extended in \cite{BG} to more general elliptic equations with Lipschitz principal part where  the authors also established a certain boundary version of the vanishing order estimate.

We mention that the proof of our Carleman estimate  is based on  elementary arguments  using  integration by parts and  an appropriate  Rellich type identity   and  is inspired by the  recent work \cite{BGM}  where a similar Carleman estimate has been established for Baouendi-Grushin operators.  Our proof however additionally exploits the discrepancy condition in \eqref{disc} below in a very crucial way. The reader will see that proof of our  Carleman estimate is based on some non-trivial geometric facts in the subelliptic setting that beautifully combine.

The paper is organized as follows. In Section \ref{pr}, we introduce some basic notations and gather some known results that are relevant to our work. In Section \ref{mn}, we prove our Carleman estimate as well as the vanishing order estimate asserted  in Theorem \ref{main} above.
\\
\textbf{Acknowledgments}

 We would like to thank Agnid Banerjee for various helpful discussions and suggestions.

\section{Notations and Preliminaries}\label{pr}
In this section we introduce the relevant notation and gather some auxiliary results that will be
useful in the rest of the paper. We will follow the same notations as in \cite{GR} and \cite{B}. For detail, we refer the reader to the book  \cite{BLU}. We now recall that a Carnot group of step $h$ is a simply connected Lie group $\mathbb{G}$ whose Lie algebra $\mathfrak{g}$ admits a stratification $\mathfrak{g}= V_1 \oplus ... \oplus V_h$ which is $h$ nilpotent, i.e., $[V_1, V_i]=V_{i+1}$ for  $i =1, ..., h-1$ and  $[V_i, V_h]=0$ for $i =1, ..., h.$ We will denote an arbitrary element of $\mathbb{G}$ by $g$ and $e$ will denote the identity of the group $\mathbb{G}.$  We will assume that $\mathfrak{g}$ is equipped with an inner product $\langle \cdot, \cdot \rangle_{\mathfrak{g}}$ such that $V_i'$s are mutually orthogonal. 

By the assumptions on the Lie algebra $\mathfrak{g},$  any basis of \textit{horizontal layer} generates the whole $\mathfrak{g}.$ Let $ \{ e_1, ...,e_m \}$ be an orthonormal basis of the first layer $V_1$ of the Lie algebra. We then define the  corresponding left invariant smooth vector fields by
\begin{equation}
	X_{i}(g)=dL_{g}(e_i), \\\  i=1, ..., m
\end{equation}	where $L_{g}$ denotes the left-translation operator given by $L_{g}(g')=gg'$ and $dL_{g} $ denote its differential. Furthur, we assume that $\mathbb{G}$ is equipped with a left inavriant Riemannian metric with respect to which $\{X_1, ...., X_m \}$ is an orthonormal set of vector fields. The sub-Laplacian corresponding to the basis  $ \{ e_1, ...,e_m \}$  is given by the formula
\begin{equation}\label{Def}
	\Delta_{H} u = \sum_{i=1}^m X_{i}^2 u.
\end{equation}
We will denote the horizontal gradient of $u$ by 
\begin{equation}
	\nabla_{H}u= \sum_{i=1}^m X_i u X_i
\end{equation}	
 and we let
\begin{equation}
	|\nabla_{H}u|^2= \sum_{i=1}^m (X_i u)^2.
\end{equation}	

We now define the anisotropic dilations $	\delta_{\lambda}$ on $\mathbb{G}$ by
\begin{equation}\label{di}
	\delta_{\lambda}(g) = exp \circ \Delta_{\lambda} \circ exp^{-1} g,
\end{equation}
where the exponential mapping $exp: \mathfrak{g} \to \mathbb{G}$ defines an analytic diffeomorphism onto $\mathbb{G}$ and for $\xi={\xi}_1+{\xi}_2+...+{\xi}_h,$ where ${\xi}_i \in V_i,$ we define
	\begin{equation}
	\Delta_{\lambda} \xi= \lambda {\xi}_1 +..... \lambda^{h} {\xi}_h,
\end{equation}
where we have assigned the formal degree $j$ to the each element of the layer $V_j.$ We will denote the infinitesimal generator of the non-isotrophic dilations \eqref{di} by $Z,$ note that such smooth vector fields is characterized by the following property 
	\begin{equation}
	\frac{d}{dr} u(\delta_{r} g) = \frac{1}{r} Zu(\delta_{r} g).
\end{equation}
Hence, $u \in C^1({\mathbb{G}})$ is a homogeneous function of degree $k$ with respect to \eqref{di}, i.e., $u(\delta_{r} g)= r^{k} u(g)$ if and only if $$Zu= ku.$$	

	 We will denote the bi-invariant Haar measure on $\mathbb{G},$ which is obtained by lifting via the exponential map \emph{exp} the Lebesgue measure on $\mathfrak{g}$ by $dg.$  Let $m_i$ denotes the dimension of $V_i.$ We then have
\begin{equation}
	(d \circ \delta_{\lambda})(g) = \lambda^{Q}dg,
\end{equation}
where
$Q= \sum_{j=1}^{h} jm_j$ is referred as the homogeneous dimension of $\mathbb{G}.$ 

Let $\Gamma (g, g')= \Gamma (g', g)$ be the positive unique fundamental solution of $-\Delta_{H}$. Note that we have $\Gamma$ is left translation invariant, i.e., 
\begin{equation}
	\Gamma (g, g')= \tilde{\Gamma}(g^{-1} \circ g')
\end{equation}
for some $\tilde{\Gamma} \in C^{\infty}(\mathbb{G}\setminus {\{e\}}).$
For every $r>0$, we define 
\begin{equation}
	B_{r}:= \left\{ g \in \mathbb{G}\; | \;\Gamma (g, e) > \frac{1}{r^{Q-2}} \right\}.
\end{equation}
 In \cite{F1}, Folland has proved that $\tilde{\Gamma}(g)$ is homogeneous function of degree $2-Q$ with respect to the non-isotrophic dilations \eqref{di}. Therefore, if we define
\begin{equation}
	\rho(g)= \tilde{\Gamma}(g)^{\frac{-1}{Q-2}}
\end{equation}
then $\rho$ is homogeneous of degree 1. Hence $B_r$ can be equivalently defined as 
\begin{equation}
	B_{r}= \{ g:\rho(g) < r \}.
\end{equation}
 We now let 
\begin{align}\label{defsi}
	\psi \overset{def} = |\n_H \rho|^2.
\end{align}
Since $\rho$ is a homogeneous function of degree 1, $\n_H \rho$ is a homogeneous function of degree 0. Hence we have
\begin{align}\label{zc}
	Z \psi =0.
\end{align}
Like in \cite{GR}, we define the discrepancy $E_{u}$  at $e$ by
\begin{equation}\label{dis}
	E_{u}= < \nabla_{H} u, \nabla_{H} \rho> - \frac{Zu}{\rho} |\nabla_{H}  \rho|^2.
\end{equation}
We will assume that  for some $\delta \in (0,1)$
\begin{equation}\label{disc}
	|E_u| \leq \frac{C_E}{\rho^{1-\delta}} |u| |\nabla_H \rho|^2.
\end{equation}
	
We  now state the following proposition from \cite{GR} concerning the action of the sub-Laplacian on radial functions.  This  will be needed in the proof of Theorem \ref{cest}. 
\begin{prop}\label{Rad}
	Let $f:(0,\infty)\to \R$ be a $C^2$ function, and define $w(g) = f(\rho(g))$. Then, one has
	\[
	\Delta_H w = |\nabla_{H} \rho|^2 \left\{f''(\rho) + \frac{Q-1}{\rho} f'(\rho)\right\},\ \ \ \ \ \text{in}\ \; \mathbb{G}\setminus \{e\}.
	\]
\end{prop}

We now collect the following elementary facts from \cite{DG} and \cite{GV}.

\begin{lemma}\label{propz}
	In a Carnot group $\mathbb{G},$ the infinitesimal generator of group dilations $Z$ enjoys the following properties:
	\begin{itemize}
		
		\item[(i)]
		One has $[X_i,Z]=X_i, \quad i=1,...,m.$
		
		\item[(ii)]
		\emph{div}$_{\bG}(\rho^{-l} Z) = (Q - l)\rho^{-l}.$
	\end{itemize}
\end{lemma}
 We need the following Rellich type identity in the proof of Theorem \ref{cest}, which corresponds to Theorem 3.1 in \cite{GV}. This can be seen as the sub-elliptic analogue of Rellich type identity  in \cite{PW}.
 \begin{lemma}
 	For a $C^{1}$ vector  field $F$,  the following holds 
 	\begin{align}\label{re}
 	&   \int_{B_r} div_{\mathbb{G}} F |\nabla_{H} v|^2 
 		 -  2\sum_{i=1}^{m}\int_{B_r} X_i v [X_i, F]v   - 2 \int_{B_r} Fv \Delta_{H} v \\
 	&	= \int_{\partial B_r} |\nabla_{H} v|^2 < F, \nu> -2\sum_{i=1}^{m} \int_{\partial B_r}Fv X_iv < X_i, \nu>.\notag
 		\end{align}
 \end{lemma} 	
We now state a Caccioppoli type energy inequality which will be used in the proof of Theorem \ref{main}. The  proof  of such an energy inequality is identical to that of  [Lemma 4.1, \cite{BGM}] and we  therefore skip the details.\begin{lemma}\label{energy}
	Let $u$   be a solution to \eqref{e0} with $V$ satisfying \eqref{vass}. Then, there exists a universal constant $C=C(Q)>0$ such that for any $0<a<1,$ we have  
	\begin{equation}\label{en}
		\int_{B_{(1-a)R} }|\nabla_{H}u|^2 \leq   \frac{C}{a^2R^2} \int_{B_R} ( 1+ K) u^2  \psi.
	\end{equation}
\end{lemma}
\

\section{Proof of Theorem \ref{cest} and \ref{main}}\label{mn}

\begin{proof}[Proof of Theorem \ref{cest}]For $R<R_0,$ let $w \in C^{2}_0(B_R\setminus \{e\}),$ satisfy the assumption in Theorem \ref{cest}. We now take $$w=\rho^{\beta} e^{-\alpha \rho^{\ve}} v,$$ where $\ve$ and $\B$ will be chosen later depending on $\delta$ and $\A$ respectively. Then we have
	\begin{align}\label{l1}
		\Delta_{H} w\;\; = \;\;v \Delta_{H} (\rho^{\beta}  e^{-\alpha \rho^{\ve}}) \;+\;2 <\nabla_{H}(\rho^{\beta} e^{-\alpha \rho^{\ve}}),\nabla_{H} v> \;+\; \rho^{\beta} e^{-\alpha \rho^{\ve}} \Delta_{H}  v.
	\end{align}
	Now we use Proposition \ref{Rad} and recall $ |\n_H \rho|^2=\psi$ to obtain
	\begin{align}\label{l2}
		&\Delta_{H}(\rho^{\beta} \ e^{-\alpha \rho^{\ve}})
		= \left(\beta(\beta+Q-2)\rho^{\beta-2}+\alpha^2 \ve^2 \rho^{\beta+2\ve-2}- \alpha\ve\left(2\beta+\ve+ Q-2\right) \rho^{\beta+\ve-2}\right) e^{-\alpha \rho^{\ve}} \psi.  
	\end{align}
Also, it is easy to check that 
	\begin{align}\label{l3}
		2 <\nabla_{H}( \rho^{\beta} e^{-\alpha \rho^{\ve}}),\nabla_{H} v>\;\;=\;\; \left(2\beta \rho^{\beta-1} -2 \ve \alpha \rho^{\beta+\ve-1}\right)  <\nabla_{H}\rho, \nabla_{H}v> e^{-\alpha \rho^{\ve}} .
	\end{align}
	Now we use \eqref{l2} and \eqref{l3} in \eqref{l1} to get
	\begin{align}\label{l4}
		\Delta_{H}w&= v\;\left(\beta(\beta+Q-2)\rho^{\beta-2}+\alpha^2 \ve^2 \rho^{\beta+2\ve-2} - \alpha\ve\left(2\beta+\ve+ Q-2\right) \rho^{\beta+\ve-2}\right) e^{-\alpha \rho^{\ve}} \psi
		\\
		& \;\;\ \ +\;\; \left(2\beta \rho^{\beta-1} -2 \ve \alpha \rho^{\beta+\ve-1}\right)  <\nabla_{H}\rho, \nabla_{H}v> e^{-\alpha \rho^{\ve}}\;+\; \rho^{\beta} e^{-\alpha \rho^{\ve}} \Delta_{H}  v .
		\notag
	\end{align}
 From definition \eqref{dis} for $E_v$, it is easy to see that \eqref{l4} is same as
	\begin{align}\label{lap}
		\Delta_{H}w + Vw&= \left(\beta(\beta+Q-2)\rho^{\beta-2}+\alpha^2 \ve^2 \rho^{\beta+2\ve-2} - \alpha\ve\left(2\beta+\ve+ Q-2\right) \rho^{\beta+\ve-2}\right) e^{-\alpha \rho^{\ve}} \psi v
		\\
		& \;\; \ \ + \left(2\beta \rho^{\beta-1} -2 \ve \alpha \rho^{\beta+\ve-1}\right)  e^{-\alpha \rho^{\ve}} \left(\frac{Zv}{\rho}\psi \; + \; E_v \right)\;+\; \rho^{\beta} e^{-\alpha \rho^{\ve}} \Delta_{H}  v +  \rho^{\beta} e^{-\alpha \rho^{\ve}}V v.
		\notag
	\end{align}
We now use the inequality $(a+b)^2\geq a^2+2ab$, with $a= 2\beta \rho^{\beta-2}  e^{-\alpha \rho^{\ve}} \psi Zv$ and $b= \Delta_{H}w +Vw -a $, where the expression for $\Delta_H w+Vw$ is given by \eqref{lap}, to find
	\begin{align}\label{c0}
		\int \rho^{-2\alpha} e^{2\alpha \rho^{\ve}}( \Delta_{H}w + Vw)^2  \psi^{-1} &\geq 4 \beta^2 \int  \rho^{2\beta-2\alpha-4}\psi (Zv)^2  +  4 \beta^2 (\beta+Q-2) \int \rho^{2\beta-2\alpha-4} \psi vZv \\
		& \ \ \ +4\alpha^2 \beta \ve^2 \int  \rho^{2\beta-2\alpha-4+2\ve}\psi vZv
		 - 4 \alpha \beta \ve(2\beta + \ve + Q -2)\int \rho^{2\beta-2\alpha-4+\ve} \psi v Zv  \notag\\ 
		 &\ \ \ - 8 \alpha \beta \ve \int \rho^{2\beta-2\alpha-4+\ve} \psi (Zv)^2    + 8\B \int \left(\beta \rho^{2\beta - 2\alpha-3} - \ve \alpha \rho^{2\beta- 2\alpha-3 + \ve} \right) E_vZv\notag\\
		 &\ \ \  + 4 \beta \int \rho^{2\beta-2\alpha-2}  \Delta_{H} v \ Zv + 4\B\int \rho^{2\B-2\A-2} Vv Zv\notag\\
		 &= I_1+I_2+I_3+I_4+I_5+I_6+I_7+I_8.\notag
	\end{align}
We now estimate each of the integrals individually. In order to estimate the $I_2,$ $I_3$ and $I_4$, first note that \eqref{zc} and (ii) in Lemma \ref{propz} gives
 \begin{align}\label{div0}
\D (\rho^{-l}\psi v^2 Z) =\psi v^2 \D(\rho^{-l}Z) + \rho^{-l}Z(\psi v^2)= (Q-l)\rho^{-l} \psi v^2+\rho^{-l} \psi Z(v^2),	
\end{align}
Also, $\text{supp}(u) \subset (B_R \setminus {\{e\}})$ Hence, \eqref{div0} gives 
\begin{align}\label{div}
	\int \rho^{-l} \psi Z(v^2) = -(Q-l)\int\rho^{-l} \psi v^2.
\end{align}
Thus using $2vZv = Z(v^2)$ and \eqref{div}, $I_2$ becomes 
\begin{align}
	4 \beta^2 (\beta+Q-2) \int \rho^{2\beta-2\alpha-4} \psi vZv &=2 \beta^2 (\beta+Q-2) \int \rho^{2\beta-2\alpha-4} \psi Z(v^2)\\
	&= -2\beta^2 (\beta+Q-2)(Q+2\B-2\A-4) \int \rho^{2\beta-2\alpha-4} \psi v^2.\notag
\end{align}
 Observe that in order to equate $I_2$ to zero, we need the following relation between $\A$ and $\B$
 \begin{align}\label{ab}
 	2\B -2\A -4 +Q=0. 
 \end{align}
Hence \begin{align}\label{i2f}
	I_2=0.
\end{align}
Again using $2vZv = Z(v^2),$ \eqref{div} and \eqref{ab} we get 
\begin{align}\label{i34f}
	I_3 + I_4 = -4\A^2 \B \ve^3 \int \rho^{-Q+2\ve} \psi v^2 + 2 \A \B \ve^2 (2\beta + \ve + Q -2) \int \rho^{-Q+\ve} \psi v^2.
\end{align}
We now estimate the integral $I_6.$ First note that using the relation \eqref{ab} and $\rho^{\ve} \le R_0^{\ve}<1$, we find 
\begin{align}\label{i31}
|I_6| =	8\B \left|\int \left(\beta \rho^{2\beta - 2\alpha-3} - \ve \alpha \rho^{2\beta- 2\alpha-3 + \ve} \right) E_vZv\right| \le 8 \beta(\B+\ve \A) \int \rho^{-Q+1}|Zv| |E_v|.
\end{align}
	In order to simplify \eqref{i31} we make use of the assumption \eqref{disc} on discrepency. To do this we observe that $E_{f(\rho)}=0$ and which furthur implies 
	\begin{align}\label{aa1}
		E_v=\rho^{-\B}e^{\A \rho^{\ve}}E_w.
	\end{align}
 Consequently, using \eqref{aa1}, \eqref{disc} and recalling $v=\rho^{-\B}e^{\A \rho^{\ve}}w$,  we deduce from\eqref{i31} 
	\begin{align}\label{i32}
	|I_6| \le	8 \beta(\B+\ve \A) C_E \int \rho^{-Q+\delta}|Zv||v|\psi.  
	\end{align}
 From \eqref{ab}, it is easy to see that for $\A>Q-4,$ we have $2\B >\A.$ Also, we have $\ve <1.$ Therefore we get $8 \beta(\B+\ve \A) \le 24 \B^2.$ Subsequently, we apply Young's equality in \eqref{i32} to find
\begin{align*}
	|I_6| \le 24\B^2 C_E \int \rho^{-Q+\delta}|Zv||v|\psi \le 12\B^2 C_E \int \rho^{-Q+\delta}|Zv|^2\psi + 12\B^2 C_E \int \rho^{-Q+\delta}|v|^2\psi.
	\end{align*}
Thus, we obtain 
\begin{align}\label{i6f}
	I_6 \ge -12\B^2 C_E \int \rho^{-Q+\delta}|Zv|^2\psi - 12\B^2 C_E \int \rho^{-Q+\delta}|v|^2\psi.
\end{align}
	Next, we simplify $I_7.$ Note that from \eqref{ab}, we have 
	\begin{equation}\label{rt2}
	I_7=	4\beta  \int  \rho^{2\beta-2\alpha-2} Zv \  \Delta_{H} v  = 4\beta  \int \rho^{-Q+2} Zv \;\Delta_{H} v.
	\end{equation}
	We now apply the Rellich type identity \eqref{re} to the vector field $ F= \rho^{-Q + 2} Z$. Also, note that since $v$ is compactly supported in $(B_{R} \setminus \{e\}),$ the boundary terms become zero. Therefore, \eqref{rt2} becomes 
	\begin{align}\label{i71}
		4 \beta \int \rho^{-Q+2} Zv \; \Delta_{H} v =2 \beta  \int \operatorname{div} (\rho^{-Q+2}Z) |\nabla_{H} v|^2   -4 \beta \sum_{i=1}^{m} \int X_i v [X_i, \rho^{-Q+2} Z] v.  
	\end{align} 
To simplify integrals in right-hand side of \eqref{i71}, recall that from (ii) of Lemma \ref{propz} we have
 \begin{align}\label{p3}
 	\operatorname{div} (\rho^{-Q+2}Z) = 2 \rho^{-Q+2}
 \end{align}
and using (i) of Lemma \ref{propz}, it is easy to obtain
\begin{align}\label{p4}
	[X_i,\rho^{-Q+2}Z]v = \rho^{-Q+2}[X_i, Z]v+X_i(\rho^{-Q+2})Zv=\rho^{-Q+2}X_iv+(2-Q) \rho^{-Q+1}X_i\rho \;Zv.
\end{align}
Consequently, using \eqref{p3} and \eqref{p4} in \eqref{i71} we find
	\begin{align}\label{ok11}
		4 \beta \int \rho^{-Q+2} Zv \; \Delta_{H} v = 4\B   \int \rho^{-Q+2} |\nabla_{H} v|^2 -4 \B \int \rho^{-Q+2}|X_iv|^2 + 4\B(Q-2) \int \rho^{-Q+1}X_i\rho X_iv \; Zv.	
	\end{align}
Since $|\nabla_{H} v|^2 = \sum_{i=1}^{m} |X_iv|^2$ and $\{X_1,X_2,...,X_m\}$ is an orthonormal set, we can rewrite \eqref{ok11} as follows
\begin{align}\label{i73}
		4 \beta \int \rho^{-Q+2} Zv \; \Delta_{H} v = 4\B (Q-2) \int \rho^{-Q+1} \langle \n_H v, \; \n_H \rho \rangle Zv.
\end{align}	
Now, we use the definition \eqref{dis} for $E_v$ in \eqref{i73} to get
\begin{align}\label{i74}
	4 \beta \int \rho^{-Q+2} Zv \; \Delta_{H} v &=4\B (Q-2) \int \rho^{-Q+1}\left( E_v + \frac{Zv}{\rho}\psi \right)Zv\notag\\
	&=4\B (Q-2) \int \rho^{-Q+1} E_v Zv + 4\B (Q-2) \int \rho^{-Q}\psi(Zv)^2.
\end{align}  	
	 We now  use \eqref{aa1} and \eqref{disc} in first integral of right-hand side of \eqref{i74} to obtain
	\begin{align*}
			4 \beta \int \rho^{-Q+2} Zv \; \Delta_{H} v \ge 
			- 4\B (Q-2)C_E \int \rho^{-Q+\delta} |v| |Zv| \psi + 4\B (Q-2) \int \rho^{-Q}\psi(Zv)^2.
	\end{align*} 
Subsequently we apply Young's inequality to get 
\begin{align}\label{i75}
		&4 \beta \int \rho^{-Q+2} Zv \; \Delta_{H} v\\& \ge 
	- 2\B (Q-2)C_E \int \rho^{-Q+\delta} \psi |v|^2  - 2\B (Q-2)C_E \int \rho^{-Q+\delta}\psi(Zv)^2 + 4\B (Q-2) \int \rho^{-Q}\psi(Zv)^2.\notag
\end{align}
 We now choose $R_0$ small enough such that $C_E R_0^{\delta} \le 1$, consequently, $C_E\rho^{\delta} <1.$ Hence \eqref{i75} becomes 
 \begin{align}\label{i7f}
 		4 \beta \int \rho^{-Q+2} Zv \; \Delta_{H} v &\ge 
 	- 2\B (Q-2)C_E \int \rho^{-Q+\delta} \psi v^2  - 2\B (Q-2) \int \rho^{-Q}\psi(Zv)^2 + 4\B (Q-2) \int \rho^{-Q}\psi(Zv)^2\notag\\
&\ge 	- 2\B (Q-2)C_E \int \rho^{-Q+\delta} \psi v^2 +2\B (Q-2) \int \rho^{-Q}\psi(Zv)^2\notag\\
& \ge - 2\B (Q-2)C_E \int \rho^{-Q+\delta} \psi v^2,
 \end{align} 
where the last inequality is a consequence of the fact that $Q\ge2.$ \\ 
We now simplify $I_8.$ 	We use the assumption \eqref{vass} followed by Young's inequality ($2AB \le A^2 +B^2$) with $A=Kv$ and $B=\B Zv$ to get
\begin{align}
	|I_8| \le 4\B \int \rho^{2\B-2\A-2}|V||v||Zv| \le 4\B K \int \rho^{2\B-2\A-2}\psi|v||Zv| \le 2 K^2 \int \rho^{2\B-2\A-2}\psi v^2 +2\B^2 \int \rho^{2\B-2\A-2}\psi|Zv|^2. 
\end{align}
 Subsequently, we use the \eqref{ab} to find
 \begin{align}\label{i8f}
 	I_8 \ge -2 K^2 \int \rho^{-Q+2}\psi v^2 -2\B^2 \int \rho^{-Q+2}\psi|Zv|^2.
 \end{align}
Therefore using \eqref{i2f}, \eqref{i34f}, \eqref{i6f}, \eqref{i7f}, \eqref{i8f}, \eqref{ab} and \eqref{c0}, for $\A>Q$ and $R_0$ small enough we have we obtain
	\begin{align}\label{c00}
		&\int \rho^{-2\alpha} e^{2\alpha \rho^{\ve}}( \Delta_{H}w+Vw)^2  \psi^{-1}\\& \geq 4 \beta^2 \int  \rho^{-Q} |Zv|^2 \psi
		  -4\A^2 \B \ve^3 \int \rho^{-Q+2\ve} \psi v^2 + 2 \A \B \ve^2 (2\beta + \ve + Q -2) \int \rho^{-Q+\ve} \psi v^2\notag \\
		  & - 8 \alpha \beta \ve \int \rho^{-Q+\ve} (Zv)^2 \psi -12\B^2 C_E \int \rho^{-Q+\delta}|Zv|^2\psi - 12\B^2 C_E \int \rho^{-Q+\delta}|v|^2\psi \notag\\
		 & - 2\B (Q-2)C_E \int \rho^{-Q+\delta} \psi v^2 -2 K^2 \int \rho^{-Q+2}\psi v^2 -2\B^2 \int \rho^{-Q+2}\psi|Zv|^2.
		\notag
	\end{align}
Now we use $\ve <1,$ $2\A>\B,$ $2\B > \A,$ which are consequences of \eqref{ab} and $\A>Q$ respectively, and rearrange the terms in right-hand side of \eqref{c00} to get 
\begin{align}\label{f1}
	\int \rho^{-2\alpha} e^{2\alpha \rho^{\ve}}( \Delta_{H}w+Vw)^2  \psi^{-1} &\geq 4 \beta^2 \int  \rho^{-Q} (Zv)^2 \psi - 16 \beta^2 \int \rho^{-Q+\ve} (Zv)^2 \psi -12\B^2 C_E \int \rho^{-Q+\delta}|Zv|^2\psi\\
	&\;\;\; -2\B^2 \int \rho^{-Q+2}\psi|Zv|^2 +2\B^3 \ve^2\int \rho^{-Q+\ve} \psi v^2 -16\B^3\int \rho^{-Q+2\ve} \psi v^2 \notag\\
	&\;\;\; - 12\B^2 C_E \int \rho^{-Q+\delta}|v|^2\psi - 4\B^2 C_E \int \rho^{-Q+\delta} \psi v^2 -2 K^2 \int \rho^{-Q+2}\psi v^2\notag
	\end{align}
At this point we would like to make the crucial observation that $-16\B^2 C_E\int \rho^{-Q+\delta} \psi v^2$ can be absorbed in the term $2\B^3\ve^2\int \rho^{-Q+\ve} \psi v^2$ provided that $\ve < \delta$ and $R_0$ is chosen small enough. Thus we now choose $\ve =\frac{\delta}{2}$ and $R_0$ small enough such that 
\begin{align}
	(18+12C_E) R_0^{\delta/2} <1\; \text{and} \;(16+16C_E) R_0^{\delta/2} <\ve^2
\end{align}
therefore we find
\begin{align}\label{11}
&	4 \beta^2 \int  \rho^{-Q} (Zv)^2 \psi - (16+12C_E+2) \beta^2R_0^{\delta/2} \int \rho^{-Q+\ve} (Zv)^2 \psi \ge 3 \B^2 \int  \rho^{-Q} (Zv)^2 \psi 
\end{align}
and \begin{align}\label{22}
	2\B^3 \ve^2\int \rho^{-Q+\ve} \psi v^2 -(16+12C_E+4C_E)\B^3R_0^{\delta/2}\int \rho^{-Q+\ve} \psi v^2 \ge \B^3 \ve^2 \int \rho^{-Q+\ve} \psi v^2.
\end{align}
Hence using \eqref{11} and \eqref{22} in \eqref{f1}, we obtain
\begin{align}\label{f2}
		\int \rho^{-2\alpha} e^{2\alpha \rho^{\ve}}( \Delta_{H}w+Vw)^2  \psi^{-1} \geq 3 \beta^2 \int  \rho^{-Q} (Zv)^2 \psi +\B^3\ve^2\int \rho^{-Q+\ve} \psi v^2-2 K^2 \int \rho^{-Q+2}\psi v^2.
\end{align}
Subsequently, if we choose $$\A >\frac{2}{\ve^{2/3}}K^{2/3}+Q$$ then from \eqref{ab}, we get $\B^3\ve^2 \ge 8K^2$. Hence \eqref{f2} becomes
\begin{align}
	\int \rho^{-2\alpha} e^{2\alpha \rho^{\ve}}( \Delta_{H}w+Vw)^2  \psi^{-1} \ge \frac{\A^3\ve^2}{16}\int \rho^{-Q+\ve} \psi v^2.
\end{align}
We now substitute $v=\rho^{-\B}e^{\A\rho^{\ve}}$ and use \eqref{ab} to get the desired estimate \eqref{est}. This completes the proof of Theorem \ref{cest}.
	\end{proof}
\begin{proof}[Proof of Theorem \ref{main}]
	We adapt arguments  from \cite{Bk1, BGM}. For a given $R_1 < R_2$, $A_{R_1, R_2}$ will denote the annulus $B_{R_1} \setminus B_{R_2}$. Let $R_0$ be as in the Theorem \ref{cest} and let $ 0 < R_1<2R_1<R_2=R_0/4.$ Also, we take a radial function $\phi \in C_{0}^{\infty}(B_{2R_2}),$ i.e., $\phi(g) =f(\rho(g))$ for some $f,$ such that
	\begin{equation}\label{bd1}
		\begin{cases}
			\phi \equiv 0\ \text{if $\rho < R_1$ and $\rho > {2R_2}$}
			\\
			\phi \equiv 1\ \text{in $A_{2R_1, R_2}$}
		\end{cases}
	\end{equation}
 and  the following bounds hold, 
	\begin{equation}\label{bd}
		|\nabla_{H} \phi| \leq  \frac{C\psi^{1/2}}{R},\ \ \ \ \ \ \ \ | \Delta_{H}\phi| \leq \frac{C \psi}{R^2}.
	\end{equation}
Without loss of generality we will assume 
\begin{align}\label{asm2}
	||  u \psi^{1/2}||_{L^2(B_{R_2})} \not=0.
\end{align}	
Since $u$ satisfies $- \Delta_{H} u =V u,$ $w=u\phi$ satisfies
	\begin{align}\label{w11}
		 \Delta_{H} w + V w = u\Delta_{H} \phi +2 \langle \n_H \phi, \n_H u \rangle.
	\end{align}
 Since $\phi$ is radial, we have $E_{\phi}=0$ and consequently, we get $E_w=\phi E_u.$ Also since $0\le \phi \le 1,$ therefore $E_w$ satisfies \eqref{disc} and moreover  by a standard limiting argument via approximation with smooth functions, we can apply the Carleman estimate in \eqref{est} to $w$.  We thus  obtain
\begin{align}\label{w2}
	\alpha^3 \int  \rho^{-2\alpha - 4 + \ve} e^{2\alpha \rho^{\ve}} u^2 \phi^2 \psi \le C \int \rho^{-2\alpha} e^{2 \alpha \rho^{\ve}}(u\Delta_{H} \phi +2 \langle \n_H \phi, \n_H u \rangle)^2\psi^{-1}.
\end{align} 
 Now we use the inequality $(a+b)^2 \le 2a^2 + 2b^2$  and Cauchy-Schwarz inequality in the right-hand side of \eqref{w2} to obtain
	\begin{equation}\label{ty11}
		\alpha^3 \int  \rho^{-2\alpha - 4 + \ve} e^{2\alpha \rho^{\ve}} u^2 \phi^2 \psi \leq 2C\int \rho^{-2\alpha} e^{2 \alpha \rho^{\ve}} (u^2 (\Delta_{H} \phi)^2 \psi^{-1}+ |\nabla_{H}u|^2 |\nabla_{H} \phi|^2 \psi^{-1}).
	\end{equation}
 For convenience, we will denote $L^{2}$ norm of $f$ in $B_R$ and $A_{R_1, R_2}$ by $||f||_R$ and $||f||_{R_1, R_2}$ respectively. 
	Note that from \eqref{bd1}, the functions $\n_H\phi$ and $\Delta_{H}\phi$ are supported in $A_{R_1,2R_1} \cup A_{R_2, 2R_2}.$ Furthur using \eqref{bd1} and \eqref{bd} in \eqref{ty11}, there exists a universal constant $C$ such that
	\begin{align}\label{bd2}
		 \A^{3/2}|| \rho^{-\alpha -2+\ve/2} e^{\alpha \rho^{\ve}} u \psi^{1/2}||_{2R_1,R_2}&\leq C\left ( || \rho^{-\alpha -2} e^{\alpha \rho^{\ve}} u \psi^{1/2}||_{R_1,2R_1} +  || \rho^{-\alpha -2} e^{\alpha \rho^{\ve}} u \psi^{1/2}||_{R_2,2R_2} \right)
		\\
		&\;\;\; + C \left( R_1|| \rho^{-\alpha -2} e^{\alpha \rho^{\ve}} |\nabla_{H}u| ||_{R_1,2R_1} + R_2 || \rho^{-\alpha -2} e^{\alpha \rho^{\ve}} |\nabla_{H}u| ||_{R_2,2R_2} \right).
		\notag
	\end{align}
We observe that the functions $$r \rightarrow r^{-\alpha-2+\ve/2} e^{\alpha r^{\ve}},\;\;\;\;\; r \rightarrow r^{-\alpha-2} e^{\alpha r^{\ve}}$$ are decreasing in $(0,1),$ therefore \eqref{bd2} gives 
\begin{align}\label{bd31}
	\A^{3/2}R_2^{-\alpha -2+\ve/2} e^{\alpha R_2^{\ve}}||  u \psi^{1/2}||_{2R_1,R_2}&\leq C\left ( R_1^{-\alpha -2} e^{\alpha R_1^{\ve}}||  u \psi^{1/2}||_{R_1,2R_1} + R_2^{-\alpha -2} e^{\alpha R_2^{\ve}} ||  u \psi^{1/2}||_{R_2,2R_2} \right)
	\\
	&\;\;\; + C \left( R_1 R_1^{-\alpha -2} e^{\alpha R_1^{\ve}}||  |\nabla_{H}u| ||_{R_1,2R_1} + R_2 R_2^{-\alpha -2} e^{\alpha R_2^{\ve}} ||  |\nabla_{H}u| ||_{R_2,2R_2} \right).
	\notag
\end{align}
From the Caccioppoli estimate in Lemma \ref{energy}, we have
\begin{equation}\label{bd5}
	\begin{cases}
		R_1  ||  |\nabla_{H}u| ||_{R_1,2R_1} \leq C(1+K^{1/2}) ||u\psi^{1/2}||_{4R_1},
		\\
		R_2|| |\nabla_{H}u| ||_{R_2,2R_2}   \leq C( 1+ K^{1/2})  || u \psi^{1/2}||_{R_0}.
	\end{cases}
\end{equation}
We now use \eqref{bd5} in \eqref{bd31} and with possibly some large universal constant $C,$ get
\begin{align}\label{bd3}
	\A^{3/2}R_2^{-\alpha -2+\ve/2} e^{\alpha R_2^{\ve}}||  u \psi^{1/2}||_{2R_1,R_2}&\leq C \big( R_1^{-\alpha -2} e^{\alpha R_1^{\ve}}||  u \psi^{1/2}||_{R_1,2R_1} + R_2^{-\alpha -2} e^{\alpha R_2^{\ve}} ||  u \psi^{1/2}||_{R_2,2R_2}
	\\
	&\;\;\; +  R_1^{-\alpha -2} e^{\alpha R_1^{\ve}}(1+K^{1/2}) ||u\psi^{1/2}||_{4R_1} + R_2^{-\alpha -2} e^{\alpha R_2^{\ve}}( 1+ K^{1/2})  || u \psi^{1/2}||_{R_0}\big).
	\notag
\end{align}
At this point, we can repeat the arguments as in the proof of Theorem 1.3 in \cite{BGM} to get to the desired conclusion.		
\end{proof}

\end{document}